\documentclass[12pt]{amsart}
\usepackage[mathscr]{eucal}
\usepackage{amssymb}
\usepackage{latexsym}
\usepackage{amsthm}

\usepackage[pagewise]{lineno}

\theoremstyle{plain}

\newtheorem*{theo}{Theorem}

\newtheorem{lem}{Lemma}[section]
\newtheorem*{cor}{Corollary}

\numberwithin{equation}{section}

\title[Hoffman's characterization theorem of parts]
{On Hoffman's characterization theorem of parts}
\author{Jun-ichi Tanaka}

\medskip

\dedicatory{Dedicated to the memory of Keiji Izuchi}

\address{Department of Mathematics,
School of Education, Waseda University,
Shinjuku, Tokyo 169-8050, Japan}

\email{jtanaka@waseda.jp}

\date{Draft was completed on July 28, 2026}


\thanks{Partially supported by NSF grant no.\,0649765}

\begin{document}
\maketitle

\begin{abstract}
Let $H^\infty(\Delta)$ be the uniform algebra of bounded analytic
functions on the open unit disc $\Delta$, and
let $\mathfrak{M}(H^\infty)$ be the maximal ideal space
of $H^\infty(\Delta)$. Applying Wermer's embedding
theorem directly, we investigate the relation between
the analytic structure in $\mathfrak{M}(H^\infty)$ and
certain separability conditions. Our method rests only on the
corona theorem and certain properties of analytic discs.  Among other
things, without deep factorization theorems on Blaschke products, we derive
the famous Hoffman theorem: Let $P(\xi)$ be the Gleason part of $\xi$ in
$\mathfrak{M}(H^\infty)$. Then $P(\xi)$ is an analytic disc if and only if
$\xi$ lies in the closure in $\mathfrak{M}(H^\infty)$ of an interpolating
sequence in $\Delta$.
\end{abstract}

\bigskip
\section{Introduction}
The theory on Gleason parts in $\mathfrak{M}(H^\infty)$,
due to K. Hoffman \cite{H3}, reveals relations
between the analytic structure and the behavior of
functions in $H^\infty(\Delta)$ around the boundary
of $\Delta$. Especially, by introducing Hoffman
maps, he characterized nontrivial
Gleason parts in $\mathfrak{M}(H^\infty)$ by interpolating sequences in $\Delta$. This result is so fundamental that many works on $\mathfrak{M}(H^\infty)$
have been exploited based on it (see, for instance, \cite{Go}, \cite{Hoe} and \cite{I}).
We identify the unit circle $\mathbf{T}$ with $[0,2\pi)$ and write $f(\theta)$ for $f(e^{i\theta})$, whenever $f$ is
a function on $\mathbf{T}$. Denote by $dm(\theta) = d\theta /2\pi$ the normalized Lebesgue measure on $\mathbf{T}$. Each function in $H^\infty(\Delta)$ is identified with its Gelfand transform. Then $H^\infty(\Delta)$ is a logmodular algebra
on the Shilov boundary $X = \mathfrak{M}(L^\infty)$,
the maximal ideal space of $L^\infty(\mathbf{T})$. Since $L^\infty(\mathbf{T})$ is identified with $C(X)$, $m$ is
lifted to a measure $\widehat{m}$ on the totally disconnected space $X$, each measurable set $E$ in $\mathbf{T}$ is represented as an open closed subset $U_E$ of $X$.

Let $\phi$ be a homomorphism in $\mathfrak{M}(H^\infty)$,
and let $\mu$ be the representing measure on $X$ for $\phi$.
For $1\le p <\infty,$ we denote by $H^p(\mu)$ the closure of $H^\infty(\Delta)$ in $L^p(\mu)$ and by $H^p_0(\mu)$
the space of all $f$ in $H^p(\mu)$ with $\phi(f)=0$.
The space $H^\infty(\mu)$ is defined to be the
weak*-closure of $H^\infty(\Delta)$ in $L^\infty(\mu)$,
which satisfies $H^\infty(\mu) = H^p(\mu) \cap
L^\infty(\mu)$. From now on, we write $B(\Delta)$ for $H^\infty(\Delta)$ to distinguish $H^\infty(\mu)$
from $H^\infty(\mu)$. Recall that a function $q$ in
$H^\infty(\mu)$ is \textit{inner} if $|q| = 1$
for $\mu -a.e.\,$, and a function $h$ in $H^1(\mu)$
is \textit{outer} if it satisfies
\begin{equation*}
\int_X\, \log \left|\,h\, \right| \,d\mu
\;=\; \log \,\left|\int_X \,h \,d\mu\,\right|
\;>\; -\infty.
\end{equation*}
These functions provide a factorization for each function
in $H^1(\mu)$ into inner and outer factors. For an $f$
in $H^1(\mu)$, let $w_n = \min(1, n/|f|)$. Since
$\log w_n$ is in $L^1(\mu)$, there is an outer function
$h_n$ such that $w_n=|h_n|$. As we set $g_n = h_n f$, the sequence $\{g_n\}$ in $H^\infty(\mu)$ satisfies that
$|g_n| \leq |f|$ and $\lim_{n\to \infty}\,
|| g_n - f||_ {L^1(\mu)} \,=\, 0$.
Moreover , $L^2(\mu)$ is a Hilbert space by the inner product
\begin{equation*}
\langle f,g \rangle \;=\; \int_X\,f\,\overline{g}\,d\mu
\qquad \text{for} \qquad f, g \in L^2(\mu).
\end{equation*}
Then $L^2(\mu) = H^2(\mu)\oplus \overline{H^2_0(\mu)}$.
A closed subspace $S$ of $H^2(\mu)$ is an \textit{invariant} subspace if the subspace $B(\Delta)\cdot S$ is contained
in $S$, where $B(\Delta)\cdot S$ may be replaced by $H^\infty(\mu)\cdot S$.

\medskip
For $\eta$ and $\xi$ in $\mathfrak{M}(H^\infty)$, the
\textit{pseudo-hyperbolic distance} $\rho(\eta,\xi)$ between
$\eta$ and $\xi$ is defined to be
\begin{equation}
\label{eq1.1}
\rho(\eta,\xi)\;=\; \sup\left\{\,\vert f(\eta)\vert \,;\;
f \in B(\Delta), \,f(\xi)=0 \; and \; \;
\Vert f \Vert_{\infty} \,\leq\, 1 \;\right\}.
\end{equation}
Then the relation $\rho(\eta,\xi)< 1$ is an equivalence relation in $\mathfrak{M}(H^\infty)$, and
the equivalence class $P(\xi) = \{ \eta \in \mathfrak{M}(H^\infty)\,; \,\rho(\eta,\xi)< 1 \}$ is called the \textit{Gleason part} of $\xi$.
A Gleason part $P$ is an \textit{analytic disc} if there
exists a continuous,
bijective map $L$ of $\Delta$ onto $P$ such that
$f\circ L$ is analytic on $\Delta$ for all $f$ in
$B(\Delta)$, and such a map is called an \textit{analytic map}.
Since $B(\Delta)$ is a logmodular algebra on $X$, it follows
from Wermer's embedding theorem that
each part is either a single point or an analytic disc.
Whenever $P$ is an analytic disc
in $\mathfrak{M}(H^\infty)\setminus \Delta$, the closure of
$P$ in $\mathfrak{M}(H^\infty)$ never meets the
Shilov boundary $X$, because of the existence of a Blaschke product vanishing identically on $P$ (see \cite[102p]{H3}).

\medskip
Furthermore, K. Hoffman \cite{H3} characterized analytic discs in $\mathfrak{M}(H^\infty)$ by using certain
sequences in $\Delta$.
Recall that a sequence $\{z_j\}$ in $\Delta$ is an
\textit{interpolating sequence} if,
for any bounded sequence $\{w_j\}$, there exists
a function $f$ in $H^\infty(\Delta)$ such that
$f(z_j)=w_j$ for $j = 1, 2, \cdots.$
Such a sequence is characterized by the condition
\begin{equation*}
\inf_{k} \prod_{j:j\not= k} \left \vert
\frac{z_j - z_k}{1 - \overline{z}_k z_j} \right \vert
\;=\; \delta\, > 0\;.
\end{equation*}
Especially, an interpolating sequence $\{z_j\}$ is said to
be \textit{thin} (or \textit{sparse}), if it satisfies
\begin{equation*}
\lim_{k\to \infty} \prod_{j:j \not= k} \left \vert
\frac{z_j - z_k}{1 - \overline{z}_k z_j}  \right \vert
\;=\;1\,.
\end{equation*}
It is known that if $\phi$ is in the closure of an thin
interpolating sequence in $\mathfrak{M}(H^\infty)$, then
the part $P(\phi)$ is homeomorphic to $\Delta$ (see
Lemma \ref{lem3.2}).
For an interpolating sequence $\{z_j\}$, the associated
Blaschke product
\begin{equation}
\label{eq1.2}
B(z)\;=\; \prod_{j= 1}^\infty \frac{\overline{z}_j}{\vert z_j\vert} \, \frac{z_j - z}{1 - \overline{z}_j z}\;,
\end{equation}
is called the \textit{interpolating\/} Blaschke product
(where $\overline{z}_j /\vert z_j\vert=-1$, if $z_j=0$\,).
Notice that a Blaschke product $B(0) > 0$, if $B(0)
\not= 0$. We sometimes use inner functions of the form $e^{i\gamma}B(z)$ for a unimodular constant $e^{i\gamma}$.

\medskip
The set $\mathfrak{M}(H^\infty)^{\Delta}$ of all maps of $\Delta$ into
$\mathfrak{M}(H^\infty)$ is a compact Hausdorff space in the product
topology. Observe that, in this topology, a net $(F_\beta)$ has limit
$F$ if and only if $\lim_\beta f\circ F_\beta(\zeta)
= f\circ F(\zeta)$ for all $f$ in $B(\Delta)$
and all $\zeta$ in $\Delta$. For an element $w$ in $\Delta$,
we put
\begin{equation}
\label{eq1.3}
L_{w}(\zeta)\;=\;
\frac{\zeta+w}{1+\overline{w}\zeta}\,,\qquad
\zeta \in \Delta\,,
\end{equation}
which is an analytic map of $\Delta$ onto the part
$\Delta$ in $\mathfrak{M}(H^\infty)$.

\medskip
Suppose that $\phi$ lies in $\mathfrak{M}(H^\infty)
\setminus \Delta$. When $\alpha$ is in ${\mathbf{T}}$,
the {\it fiber\/} $\mathfrak{M}_{\alpha}$ of
$\mathfrak{M}(H^\infty)$ over $\alpha$ is defined to be
\begin{equation*}
\mathfrak{M}_{\alpha}\;=\;\left\{\xi\in
\mathfrak{M}(H^\infty) \;; \; \xi(z)\,=\,\alpha \right\},
\end{equation*}
where $z$ is the coordinate function. We then have the
decomposition
\begin{equation*}
\mathfrak{M}(H^\infty)\setminus \Delta\;=\;
\bigcup_{\vert\alpha\vert = 1}\,\mathfrak{M}_{\alpha}.
\end{equation*}
Since each $\mathfrak{M}_{\alpha}$ is a peak set with
peaking function $(1+\bar{\alpha}z)/2$, if $\xi$ is
in $\mathfrak{M}_{\alpha}$, then its Gleason part
$P(\xi)$ is contained in $\mathfrak{M}_{\alpha}$.
Since various fibers are homeomorphic to one another,
we restrict our attention to the fiber $\mathfrak{M}_1$
over $z=1$ to look into the structure of fringe
$\mathfrak{M}(H^\infty)\setminus \Delta$. On the other
hand, the corona theorem assures the existence of a net
$(w_\alpha)$ converging to $\phi$, by identify $w$
in $\Delta$ with the homomorphism `` evaluation at
$w$.''  Taking a finer subnet, we may assume the net
$(L_{w_\alpha})$ converges
to $L$ in $\mathfrak{M}(H^\infty)^{\Delta}$ with
$L(0) = \phi$. The map $L\,=\,\lim_\alpha L_{w_\alpha}$ is called the \textit{Hoffman map} induced by $\{w_\alpha\}$,
which is an analytic map of $\Delta$ to the Gleason part
$P(\phi)$ of $\phi$. We note here that if $f$ lies in
$B(X)$, then $\lim_\alpha f\circ L_{w_\alpha} = f\circ L$ uniformly on compact subsets of $\Delta$. Otherwise, for an $\varepsilon_0 > 0 $ and for a compact subset $K$ of $\Delta$,
there is a subnet $(w_\beta)$ of $(w_\alpha)$ such that
\begin{equation*}
\max_K \; \vert \,f\circ L_{w_\beta}(\zeta) -
f\circ L(\zeta) \vert\;\geq\; \varepsilon_0,
\end{equation*}
from which we have a contradiction, because $\{f\circ L_{w_\beta}\}$ forms a normal family.
In what follows, we sometimes consider the Hoffman map
induced by a converging subnet $(w_\alpha)$ of an
interpolating sequence, for which $(L_{w_\alpha})$ converges
to $L$ automatically (see \cite[Chapter X, Lemma 1.3]{Ga}).
Such Hoffman maps play an important role to investigate analytic discs in $\mathfrak{M}(H^\infty)$.

\begin{theo}
Let $\phi$ be a homomorphism in $\mathfrak{M}(H^\infty)
\setminus \Delta$ with representing measure $\mu$ on
$X$, and let $(w_\alpha), L=\lim_\alpha L_{w_\alpha}$
and $P(\phi)$ be as above. Then the following properties
are equivalent:

\begin{enumerate}
\renewcommand{\labelenumi}{(\alph{enumi})}

\medskip
\item $P(\phi)$ is an analytic disc, that is,
  a nontrivial Gleason part in
  $\mathfrak{M}(H^\infty)$.

\medskip
  \item We may choose a sequence $\{\phi^{(m)}\}$
  of homomorphisms in $\mathfrak{M}(H^\infty)$
  such that each $\phi^{(m)}$ lies in the closure
  of thin interpolating sequence $\{w_i^m\}$ with
  the following property: The above $L$ is
  an analytic map of $\Delta$ onto $P(\phi)$
  satisfying that, for all $f$ in $B(\Delta)$,
  \begin{equation*}
  \lim_{m\to \infty}\,f\circ L^{(m)}(\zeta) =
  f\circ L(\zeta)
  \end{equation*}
  uniformly on compact subsets of $\Delta$, where
  $L^{(m)}$ denotes the Hoffman map induced by a
  subnet $(w_\alpha^m)$ of $\{w_i^m\}$ converging
  to $\phi^{(m)}$.  Succinctly, $P(\phi)$ is
  approximated by a sequence $P(\phi^{(m)})$ of
  homeomorphic parts in $\mathfrak{M}(H^\infty)$.

\medskip
  \item There exists an inner function $F$ in
  $H^2(\mu)$ such that $\int_X F \,d\mu = 0\,$
  with the following property: Let $\mathcal{H} =
  \mathcal{H}(F)$ be the closed subspace generated
  by $\{F^{n}; \;n=0, 1, 2, \cdots\}$ in $H^2(\mu)$,
  and let
  $\mathcal{S} = \mathcal{S}(F)$ be the complementary subspace of $\mathcal{H}$ in $H^2(\mu)$, that is,
  $H^2(\mu) = \mathcal{H}\oplus\mathcal{S}$. Then
  $\mathcal{S}$ is an invariant subspace of $H^2(\mu)$.  In this case, if $h$ is in $S$, then
  \begin{equation*}
  \int_X \, h \,\frac{1-|\zeta|^2}{|F -\zeta|^2}
  \,d\mu \;=\; 0\,,
  \end{equation*}
  for all $\zeta$ in $\Delta$ (compare with
  Lemma \ref{lem2.2}).

\medskip
  \item There exists an interpolating sequence
  $\{c_i\}$ in $\Delta$ such that, for a suitable
  subnet $(c_\alpha)$ of $\{c_i\}$, the Hoffman map $L=\lim_\alpha L_{c_\alpha}$ with $L(0) = \phi$
  satisfies that $P(\phi)=L(\Delta)$.
\end{enumerate}
\end{theo}

\medskip
The equivalence (a) and (d) is a strong version of
Hoffman's theorem :

\begin{cor}
Let $P(\phi)$ be a Gleason part in $\mathfrak{M}
(H^\infty)\setminus \Delta$. Then $P(\phi)$ is an
analytic disc if and only if $\phi$ lies in the
closure in $\mathfrak{M}(H^\infty)$ of an
interpolating sequence in $\Delta$.
\end{cor}

\medskip
In the next section, we establish some notation and elementary facts derived from Wermer's embedding theorem. In Section 3,
among other things, Hoffman maps are discussed by several lemmas in relation to interpolating sequences, and the proof
of Theorem is also provided in its last part. We close with two remarks in Section 4.

\medskip
We refer the reader to \cite{H3} and \cite[Chapter X]{Ga}
for further details and recent developments on the analytic structure of $\mathfrak{M}(H^\infty)$. Related results concerning the Hardy space theory can be found in
\cite{G1}, \cite{Ga} and \cite{H1}.


\bigskip
\section{Embedding theorem and approximation}
\label{S2}

We recall certain properties of inner functions appearing
in Wermer's embedding theorem. Let $\phi$ be a homomorphism
in $\mathfrak{M}(H^\infty)$, and let $\mu$ be the
representing measure for $\phi$ on $X$. Since
$B(\Delta)=H^\infty(\Delta)$ is a logmodular algebra on
$X$, $\mu$ is determined uniquely, from which we know
that each Gleason part is either a single point or an
analytic disc. Suppose that the Gleason part $P(\phi)$
of $\phi$ is an analytic disc. Then each $\xi$ in $P(\phi)$
has a representing measure $\mu_\xi$ which is mutually
absolutely continuous with respect to $\mu$. Furthmore,
the Radon-Nikodym derivatives $d\mu /d\mu_{\xi}$ and
$d\mu_{\xi} /d\mu$ are essentially bounded. Then Wermer's
embedding theorem assures the existence of an inner
function $Z$ in $H^\infty(\mu)$ such that its extension
$\widehat{Z}(\xi) \,= \, \int_X\,Z\,d\mu_\xi$ to
$P(\phi)$ is a bijective map of $P(\phi)$ to $\Delta$.
Then $\tau(\zeta)= \widehat{Z}^{-1}(\zeta)$ is an analytic
map on $\Delta$ with $\tau(0) =\phi$. We notice that $Z$
is determined uniquely up to multiplication by unimodular
constants. Moreover, for each $f$ in $B(\Delta)$, $f$ is
represented as
\begin{equation}
\label{eq2.1}
f(\xi)\;=\;\sum_{n=0}^{\infty}\, a_n \widehat{Z}(\xi)^n\;,
\qquad \xi \in P(\phi)\,,
\end{equation}
where $a_n = \langle f,Z^{\,n}\rangle$, the inner
product of $f$ and $Z^n$ in $L^2(\mu)$ (see,
for example, \cite[Chapter 6, \S 6.4]{L}\,). Since
$\widehat{Z}\circ\tau(\zeta)=\zeta$, we observe that
$f\circ\tau(\zeta)\,=\,\sum_{n=0}^{\infty}\, a_n \zeta^n$.
Since $f\circ\tau$ is a bounded analytic function on
$\Delta$. the sequence $\{a_n\}$ of its Taylor coefficients
lies in $\ell^2 = \ell^2(0, 1, \ldots)$. Observe also that
$a_0=(f\circ\tau)(0)$ and $a_1=(f\circ\tau)'(0)$. However,
since the orthonormal system $\{Z^{n}; \;n=0,\pm1, \pm2,
\cdots\}$ may not be complete in $L^2(\mu)$, we only have
the inequality
\begin{equation*}
|| f ||^2_{L^2(\mu)}
\;\geq \; \sum_{n=0}^{\infty}\, |a_n|^2.
\end{equation*}
Indeed, as we mentioned earlier, if $P(\phi)$ is contained
in $\mathfrak{M}(H^\infty)\setminus \Delta$, then there is
a Blaschke product $B$ vanishing identically on $P(\phi)$.
This shows that $B$ is orthogonal to the above system
$\{Z^{n}\}$ in $L^2(\mu)$. We use repeatedly the above
notation in forthcoming sections.

\medskip
The next lemma is a variant of the Hoffman-Wermer theorem
(see \cite[Chapter II, Theorem 7.2]{G1} or \cite[Chapter 6,
Theorem 26 ]{L}\,):

\begin{lem}
\label{lem2.1}
Under the above notation, there is a sequence $\{q_n\}$
in $B(\Delta)$ with $q_n(\phi) = q_n\circ\tau(0) = 0$
such that $\Vert q_n\Vert_\infty \leq 1$ and
$q_n \to Z $ for $\mu-a.e.\,$ on $X$. Moreover, we
may assume that $(q_n\circ\tau)'(0) > 0$.
\end{lem}

\begin{proof}\;
Since $H^\infty(\mu)= H^1(\mu)\cap L^\infty(\mu)$, there
is a sequence $\{g_n\}$ in $B(\Delta)$ such that
$\Vert g_n-Z\Vert_{L^1(\mu)} < 1/n^2$. Since $g_n(\phi)
\to Z(\phi) = 0$, we may assume $g_n(\phi) = 0$.
Replacing with its suitable subsequence, we also
assume $g_n \to Z$ for $\mu-a.e.\,$ on $X$.
It suffices to show that there is a sequence
$\{h_n\}$ in $B(\Delta)$ such that $\Vert g_n e^{-h_n} \Vert_\infty \leq 1$ and $e^{-h_n} \to 1$ for $\mu-a.e.$
on $X$. Since $\vert Z\vert = 1$,
$\log^+ \vert g_n\vert \leq \max(0, \vert g_n\vert
-1) \leq \vert g_n-Z\vert$. This shows
\begin{equation*}
\int_X \log^+ \vert g_n\vert \,d\mu \;\leq \;
\int_X \vert g_n-Z\vert \,d\mu \;<\;
\frac{1}{n^2}\,.
\end{equation*}
Since the representing measure $\mu$ for $\phi$
is unique, we have
\begin{equation*}
\int_X \log^+ \vert g_n\vert \,d\mu = \inf\{
\phi(v);\;\; \log^+ \vert g_n\vert
\leq v, \;\; v \in Re \,B(X)\,\}.
\end{equation*}
This shows that there is a function $h_n$ in
$B(X)$ with $Im\,(h_n)(0)=0$ such that
$\log^+ \vert g_n\vert \le Re\,(h_n)$, while
$Re\,\phi(h_n)\le 1/n^2$. By the same way as in
the proof of \cite[Chapter II, Theorem 7.2]{G1},
we see that $q_n=g_n e^{-h_n}$ satisfies the
desired property.

Moreover, since $\Vert q_n-Z\Vert_{L^1(\mu)}
\to 0$, we see that $q_n\circ\tau(\zeta) \to
Z\circ\tau(\zeta) = \zeta$, so $(q_n\circ\tau)'(0)
\to 1 $. Let $\{e^{i\gamma_n}\}$ be the sequence
such that $e^{i\gamma_n}g_n\circ\tau)'(0) > 0$.
by replacing $q_n$ with $e^{i\gamma_n}g_n$,
$q_n$ satisfies the desired property.

\end{proof}

\medskip
Observe that if $z_1$ and $z_2$ lie in $\Delta$, then
Schwarz's lemma shows that the
psudo-hyparbolic distance $\rho(z_1,z_2)$ by
\eqref{eq1.1} is given by
\begin{equation}
\label{eq2.2}
\rho(z_1,z_2)\;=\;\left \vert
\frac{z_1 - z_2}{1 - \overline{z}_2 z_1} \right \vert\,.
\end{equation}

\begin{lem}
\label{lem2.2}
Under the above notation, we obtain the following :

\begin{enumerate}
\renewcommand{\labelenumi}{(\alph{enumi})}

\item
  For $\xi$ in $P(\phi)$, the representing
  measure $\mu_\xi$ has the form
  \begin{equation*}
  d\mu_\xi
  \;=\; \frac{1-|\widehat{Z}(\xi)|^2}
  {|Z -\widehat{Z}(\xi)|^2}\,d\mu
  \;=\; Re\left\{\frac{Z +\widehat{Z}(\xi)}
  {Z -\widehat{Z}(\xi)}\right\}\,d\mu\,.
  \end{equation*}

\item
  For $\xi_1$ and $\xi_2$ in $P(\phi)$, we have
\begin{equation*}
\rho(\xi_1,\xi_2)\;=\;
\rho(\widehat{Z}(\xi_1),\widehat{Z}(\xi_2))\;=\;
\left|\frac{\widehat{Z}(\xi_1)-\widehat{Z}(\xi_2)}
{1- \overline{\widehat{Z}(\xi_2)}\,\widehat{Z}(\xi_1)}
\right |.
\end{equation*}
\end{enumerate}

\end{lem}

\begin{proof}\;
By putting $F=Z$ and $\lambda=\widehat{Z}(\xi)$, the
property (a) follows from \cite[Chapter V, Theorem 7.1]
{G1}, whose proof rests the fact that $(\lambda -
Z)H^\infty(\mu)$ has codimension one in $H^\infty(\mu)$
for each $\lambda$ in $\Delta$. This follows from
$H^\infty(\mu) = H^\infty(\mu_\xi)$ and
\begin{equation*}
H^\infty_0(\mu_\xi) \;=\; \frac{Z-\lambda}{1-\overline{\lambda} Z}\,H^\infty(\mu_\xi) \;=\; (\lambda - Z)H^\infty(\mu).
\end{equation*}

To show the property (b), let $g$ be a function in $B(\Delta)$ such that $||g||_\infty \leq 1$ and $g(\xi_2)=0$. Since $g\circ \tau$ is analytic on $\Delta$ and $g(\xi)=g\circ \tau(\widehat{Z}(\xi))$ on $P(\phi)$,
Schwarz's inequality implies that
\begin{equation*}
\left|\frac{\widehat{Z}(\xi_1)-\widehat{Z}(\xi_2)}
{1-\overline{\widehat{Z}(\xi_2)}\,\widehat{Z}(\xi_1)}
\right | \; \geq \;
\left|\frac{g\circ \tau(\widehat{Z}(\xi_1))-
g\circ \tau(\widehat{Z}(\xi_2))}
{1-\overline{g\circ \tau(\widehat{Z}(\xi_2))}\,g\circ \tau(\widehat{Z}(\xi_1))}\right |\;=\;
\left|g(\xi_1)\right|\,.
\end{equation*}
So we have $\rho(\widehat{Z}(\xi_1),\widehat{Z}(\xi_2)) \,\geq\,\rho(\xi_1,\xi_2)$ by \eqref{eq1.1}.
Conversely, let $\{q_n\}$ with $|q_n| \leq 1$ be
an approximating sequence to $Z$ as in Lemma \ref{lem2.1}. Then
\begin{equation*}
\rho(\xi_1,\xi_2)\;\geq\;
\left|\frac{q_n(\xi_1)-q_n(\xi_2)}
{1-\overline{q_n(\xi_2)}\,q_n(\xi_1)}
\right | \;\rightarrow \;
\left|\frac{\widehat{Z}(\xi_1)-\widehat{Z}(\xi_2)}
{1-\overline{\widehat{Z}(\xi_2)}\,\widehat{Z}(\xi_1)}
\right |\,, \quad \text{as} \quad  n \to \infty\,.
\end{equation*}
This shows that $\rho(\xi_1,\xi_2)\;\geq\; \rho(\widehat{Z}(\xi_1),\widehat{Z}(\xi_2))$,
thus the equality of (b) holds.
\end{proof}

\begin{lem}
\label{lem2.3}
Under the above notation, let $\mathcal{H}=\mathcal{H}(Z)$ and
$\mathcal{S}=\mathcal{S}(Z)$ be the subspaces of $H^2(\mu)$ as in (c) of Theorem with $F$ replaced by $Z$.
Then we have the following properties;

\begin{enumerate}
\renewcommand{\labelenumi}{(\alph{enumi})}

\item
  A function $h$ in $H^2(\mu)$ lies in $\mathcal{S}$
  if and only if $h(\xi) \,=\, 0$ for all $\xi$ in
  $P(\phi)$;

\medskip
\item
  $\mathcal{S}$ is an invariant subspace of $H^2(\mu)$.
\end{enumerate}
\end{lem}

\begin{proof}\; To show the property (a), let $h$ be a
function in $\mathcal{S}$.
Since $h$ is orthogonal to $\mathcal{H}$, we see that
$\langle h, Z^n \rangle \,=\, 0$ for $n=0, 1, \cdots$.
On the other hand, since $h$ lies in $H^2_0(\mu)$, we have
$\langle h , Z^n \rangle \,=\, 0$ for $n=-1, -2, \cdots$.
It follows from (a) of Lemma \ref{lem2.2} that $h(\xi)
= \int_X h \,d\mu_\xi = 0$ for all $\xi$ in $P(\phi)$.
Conversely, suppose for the contrary that a function
$h$ satisfies that $h(\xi) = 0$ on $P(\phi)$ and
$a_n = \langle h , Z^n \rangle \,\neq\,0$ for some
$n \geq 0$. However, \eqref{eq2.1} shows that
$h(\xi)$ may not be identically $0$ on $P(\phi)$,
which is a contradiction. To show (b), let $h$ be a
function in $\mathcal{S}$ and let $f$ be a function
in $B(X)$. Since $\mu_\xi$ is also multiplicative
on $H^2(\mu) = H^2(\mu_\xi)$, we have $(fh)(\xi)
= f(\xi)h(\xi) =0$ for all $\xi$ in $P(\phi)$.
Then (a) shows that $fh$ lies in $\mathcal{S}$,
thus the property (b) holds.
\end{proof}

\bigskip
\section{Characterization of analytic discs}
\label{S3}

We begin this section by showing certain properties of
thin interpolating sequences. The following is a strong
version of \cite[204p, Corollary]{H3} by a similar argument.

\begin{lem}
\label{lem3.1}
Any sequence $\{z_n\}$ in $\Delta$ such that
$\limsup_n \vert z_n\vert = 1$ contains a subsequence
$\{c_k\}$ which is a thin interpolating sequence.
\end{lem}

\begin{proof}
\; Let $\{\varepsilon_n\}$ be a sequence of positive numbers
with $\lim_n \varepsilon_n = 0$, and let $\{r_n\}$ be a
decreasing sequence of positive numbers
such that
\begin{equation*}
\prod_{j=1}^\infty \frac{1-r_n^j}{1+r_n^j}
\;>\;\sqrt{1-\varepsilon_n}\,,\quad n=1, 2, \cdots\,.
\end{equation*}
Replacing $\{z_n\}$  with a suitable subsequence,
we assume that
\begin{equation*}
\frac{1-\vert z_n\vert}{1-\vert z_{n-1}\vert}
\;<\; r_{n-1}.
\end{equation*}
When $j>k$, we have
\begin{equation*}
1-\vert z_j\vert \leq r_{j-1}r_{j-2}
\cdots\,r_k (1-\vert z_k\vert)\leqq r_k^{j-k}(1-\vert z_k\vert)
\end{equation*}
and thus
\begin{equation*}
\vert z_j\vert-\vert z_k\vert
\geq (1-r_k^{j-k})(1-\vert z_k\vert)\,.
\end{equation*}
On the other hand,
\begin{equation*}
1-\vert z_j\vert \vert z_k\vert
= (1-\vert z_k\vert)-(\vert z_j\vert \vert z_k\vert
-\vert z_k\vert))
\leqq (1+r_k^{j-k})(1-\vert z_k\vert)\,.
\end{equation*}
As in the proof of the Hayman and Newman theorem
(\cite[203p]{H3}), we obtain
\begin{equation}
\label{eq3.1}
\prod_{j= k+1}^\infty \left \vert
\frac{z_j - z_k}{1 - \overline{z}_k z_j}  \right \vert
\;\geq \;
\prod_{j=1}^\infty \frac{1-r_k^j}{1+r_k^j} \;>\;\sqrt{1-\varepsilon_k}\,.
\end{equation}
Now let us choose a thin subsequence $\{c_k\}$
by induction. Put $c_1 = z_1$ and suppose
$c_1, c_2, \cdots, c_{k-1}$ have been chosen. Find
$c_k = z_{k+m}$ in $\{z_j; j \geq k\}$ such that
\begin{equation*}
\prod_{j= 1}^{k-1}\left \vert
\frac{c_k - c_j}{1 - \overline{c_j}c_k}  \right \vert
\;>\;\sqrt{1-\varepsilon_k}\,.
\end{equation*}
Since we have
\begin{equation*}
\prod_{j= k+m+1}^\infty \left \vert
\frac{z_j - c_k}{1 - \overline{c_k}z_j} \right \vert
\;>\;\sqrt{1-\varepsilon_{k+m}}\,>\,\sqrt{1-\varepsilon_k}\,
\end{equation*}
by \eqref{eq3.1}, it is easy to see that $\{c_k\}$
is a thin subsequence of $\{z_n\}$.
\end{proof}

\medskip
There is another way to show the existence of thin
interpolating subsequences. For given $\varepsilon$
and $\eta$ with $0 < \varepsilon,\eta <1$, there is
a $\delta=\delta(\varepsilon, \,\eta) > 0$ such that
if the zeros $\{z_k\}$ of a Blaschke product $B(z)$
satisfy that
$\sum_{k=1}^\infty (1-\vert z_k\vert) < \delta,$
then $\vert B(z)\vert > 1-\varepsilon$ for
$\vert z \vert \leq\eta$ (see \cite[Lemma 4.1]{T2}).
From this fact we may choose easily a thin subsequence
from any interpolating sequence.

\medskip
The following is well-known (see \cite[Chapter X, Exercise
8]{Ga}). However, for our purposes, it would be helpful
to state it precisely. For an analytic disc $P(\phi)$ in
$\mathfrak{M}_1$, let $Z, \widehat{Z}, \tau =
\widehat{Z}^{-1}$ and $\{q_n\}$ be as in Section 2,
associated with Wermer's embedding theorem.

\begin{lem}
\label{lem3.2}
Let $\{c_n\}$ be a thin interpolating sequence in $\Delta$
with $\lim_n c_n = 1$, and let $B$ be the Blashke product
with zeros $\{c_n\}$. Denote by $L$ be the Hoffman map
$L=\lim_\beta L_{c_\beta}$ induced by a convergent subnet
$(c_\beta)$ of $\{c_n\}$ in $\mathfrak{M}(H^\infty)$. If
$P(\phi)$ is the Gleason part of $\phi=L(0)$, then $P(\phi)$
is an analytic disc in the fiber $\mathfrak{M}_1$, and $L$
is a homeomorphism of $\Delta$ onto $P(\phi)$ such that
$(e^{i\alpha}B\circ L)(\zeta) = \zeta$ with unimodular
constant $e^{i\alpha}$.
\end{lem}

\begin{proof}
\; Since if $f$ is in $B(\Delta)$, then $f\circ L =
\lim_\beta f\circ L_{c_\beta }$ uniformly on compact
subsets of $\Delta$, $L$ is an analytic map of $\Delta$
to $P(\phi)$ by \cite[Chapter X, Lemma 1.1]{Ga}. Hence,
$L$ is continuous on $\Delta$. Observe that $B(0)=\Pi_{n=1}^\infty\,|c_n|\,\geq 0$
and $(B\circ L)(0) = B(\phi) = 0$. Since
\begin{equation*}
\vert(B\circ L)'(0)\vert \;=\; \lim_\beta\,
\vert(B\circ L_\beta)'(0)\vert
 \;=\; \lim_\beta \,(1-\vert c_\beta\vert^2)\,
 \vert B'(c_\beta)\vert \,=\,1,
\end{equation*}
$P(\phi)$ is a nontrivial part and Schwarz's
lemma shows that $(B\circ L)(\zeta) = e^{-i\alpha}\zeta$,
for an unimodular constant $e^{i\alpha}$. Putting
$B_0(\zeta)= e^{i\alpha}B(\zeta)$, we obtain
$B_0\circ L)(\zeta)=\zeta$. For the analytic disc
$P(\phi)$, Wermer's embedding theorem shows that
the above $\tau=\widehat{Z}^{-1}$ is an analytic map
of $\Delta$ onto $P(\phi)$ with $\tau(0)=\phi$. Let
$f=\widehat{Z}\circ L$ and $g=B_0\circ \tau$. Then
both functions map $\Delta$ into itself, and vanish
at $0$. Since $g(f(\zeta)) = (B_0\circ L)(\zeta)
= \zeta$, we obtain $g'(f(\zeta))f'(\zeta) \equiv 1$,
so $g'(0)f'(0)=1$. Since $f'(0)$ is a unimodular
constant we see that $L$ maps $\Delta$ onto $P(\phi)$.
Since $B_0=L^{-1}$ is continuous on $P(\phi)$, $L$ is
a homeomorphism, as desired.
\end{proof}

Recall that $Z$ is determined up to multiplication by
unimodular constants. Since $f'(0)= e^{i\beta}$,
the map $\tau$ is identified with $L$, by replacing
$Z$ with $e^{-i\beta}Z$, This property partially
extends to any nontrivial part.

\medskip
Let $\phi$ be a homomorphism in the fiber $\mathfrak{M}_1$
of which $P(\phi)$ is an analytic disc. Then the corona
theorem assures the existence of a net $(w_\alpha)$ in
$\Delta$ such that $\lim_\alpha w_\alpha = \phi$ in
$\mathfrak{M}(H^\infty)$. Let $L_{w_\alpha}(\zeta) =
(\zeta+w_\alpha)/(1+\overline{w}_\alpha \zeta)$ as
before. Replacing a suitable subnet of $(L_{w_\alpha})$,
we assume also that there is an $L$ in $\mathfrak{M}
(H^\infty)^\Delta$ with $L(0) = \phi$ such that
$L\,=\,\lim_\alpha L_{w_\alpha}$.
Denote by $(w_\alpha)^*$ the set of all points in
$\Delta$ appearing in the net $(w_\alpha)$. For the
sequence $\{q_n\}$ in Lemma \ref{lem2.1}, we consider
the finite subset $\{q_1, q_2, \cdots, q_m\}$ of
$\{q_n\}$. It follows from the definition of nets
(directed sets) that there is a thin interpolating
sequence $\{w_i^m\}$ in $(w_\alpha)^*$ such that
\begin{equation}
\label{eq3.2}
q_k\circ L(\zeta)\;=\; \lim_{i\to \infty}\, q_k\circ
L_{w_i^m}(\zeta)\,, \qquad \text{for}\qquad
k=1, 2, \cdots, m,
\end{equation}
uniformly on compact subsets of $\Delta$. Let
$L_i^{(m)} = L_{w^m_i}$, and let $L^{(m)}$ be the Hoffman
map $L^{(m)}=\lim_\beta L_\beta^{(m)}$ by a convergent
subnet $(L_\beta^{(m)})$ of the sequence $\{L_i^{(m)}\}$
in $\mathfrak{M}(H^\infty)^\Delta$. Since $\{w^m_i\}$ is
a thin interpolating sequence, Lemma \ref{lem3.2} shows
that $L^{(m)}$ is a homeomorphism of $\Delta$ onto
$P(\phi^{(m)})$ with $\phi^{(m)} = L^{(m)}(0)$. Let
$B^{(m)}$ be the Blaschke product with zeros $\{w^m_i\}$.
Then $B^{(m)}$ satisfies that $e^{i\alpha_m} B^{(m)}
\circ L^{(m)}(\zeta) =\zeta$ with a unimodular constant
$e^{i\alpha_m}$.

On the other hand, by Wermer's embedding theorem on
$P(\phi)$, we know that each $q_k$ is represented as
\begin{equation}
\label{eq3.3}
q_k(\xi)
= a_1^{(k)}\widehat{Z}(\xi) + a_2^{(k)}\widehat{Z}(\xi)^2 +
\cdots \,,
\qquad \xi \in P(\phi),
\end{equation}
with $\{a_i^{(k)}\}$ in $\ell^2$. We assume that
$(q_k\circ \tau)(0) = a_0^{(k)} = 0$ and
$(q_k\circ \tau)'(0) = a_1^{(k)} > 0$. This assumption
shows that $Z, \widehat{Z}$ and $ \tau = \widehat{Z}^{-1}$
are determined uniquely. Notice also that $a_1^{(k)} \to 1$
as $k \to \infty$, by the property of $\{q_n\}$ in Lemma
\ref{lem2.1}.

\medskip
In the next lemma, we discuss the role of the Hoffman map
$L$ in relation to the above map $\tau$ on $\Delta$;

\medskip
\begin{lem}
\label{lem3.3}
For an analytic disc $P(\phi)$ in $\mathfrak{M}_1$, let
$Z, \widehat{Z}, \tau = \widehat{Z}^{-1}$ and
$L\,=\,\lim_\alpha L_{w_\alpha}$ be as above. Then we
may choose a unimodular constant $e^{i\alpha}$ such that
\begin{equation}
\label{eq3.4}
\widehat{Z}\circ L(\zeta) = e^{i\alpha} \zeta\,,
 \qquad \zeta\in \Delta\,,
\end{equation}
so that $(e^{-i\alpha}\widehat{Z})^{-1}(\zeta) = L(\zeta)$.
Consequently, in Wermer's embedding theorem
on $P(\phi)$, if we replace $Z$ with $e^{-i\alpha} Z$,
then $\tau = L$, which is a continuous
map of $\Delta$ onto $P(\phi)$.
\end{lem}

\begin{proof}
Since $\{w_i^m\}$ is a thin interpolating sequence,
the above $L^{(m)}$ is an analytic homeomorphism of
$\Delta$ onto $P(\phi^{(m)})$. Although $L^{(m)}$
is usually different from $L$, it follows from
\eqref{eq3.2} that
\begin{equation}
q_k\circ L^{(m)}(\zeta) \;=\; q_k\circ L(\zeta),
\qquad \text{for}\quad k=1, 2, \cdots, m.
\label{eq3.5}
\end{equation}
So we see that
\begin{equation*}
\lim_{m\to \infty}\,q_k\circ L^{(m)}(\zeta)
\,= \,q_k\circ L(\zeta)
\end{equation*}
uniformly on on $\Delta$, for all $k=1, 2, \cdots$.
Then \eqref{eq3.5} implies also that
\begin{equation}
\label{eq3.6}
(q_k\circ L^{(m)})'(0)\;=\;(q_k\circ L)'(0) \qquad
\text{for} \quad m= k, k+1, k+2, \cdots.
\end{equation}
Since $\vert q_k\circ L (\zeta)\vert \leq 1$ and
$q_k\circ L (0) = 0$, Schwarz's lemma shows that
$|(q_k\circ L)'(0) |\leq 1$.
Observe that $q_k\circ L^{(m)}$ is a bounded analytic
function on $\Delta$ with $q_k\circ L^{(m)}(0) = 0$,
and recall that $(e^{i\alpha_m} B^{(m)} \circ L^{(m)})
(\zeta) = \zeta$. Then there is a sequence $\{b^{(k)}_n \}$
in $\ell^2$ such that
\begin{equation*}
q_k\circ L^{(m)}(\zeta)
\;=\; \sum_{n=1}^\infty \,b^{(k)}_n \zeta^n
\;=\;\sum_{n=1}^\infty \, b^{(k)}_n \,
(e^{i\alpha_m} B^{(m)}\circ L^{(m)})^n(\zeta),
\qquad \zeta \in \Delta,
\end{equation*}
which represents also $q_k\circ L(\zeta)$ by
\eqref{eq3.5}. We now write $b^{(k)}_1 e^{-i\alpha}
\,=\,|b^{(k)}_1|\,$ and put $a^{(k)}_n
= b^{(k)}_n e^{-i n \alpha}$ in \eqref{eq3.3}.
Since $\widehat{Z}\circ \tau(\zeta) = \zeta$ and
$q_k\circ L(\zeta) \,=\, \sum_{n=1}^\infty \,
a^{(k)}_n (\widehat{Z}\circ L)^n(\zeta)$ by
\eqref{eq3.3}, we have
\begin{equation*}
\begin{split}
q_k\circ L(\zeta)
&= \sum_{n=1}^\infty \,
b^{(k)}_n \zeta^n
\;=\;\sum_{n=1}^\infty \,
b^{(k)}_n (\widehat{Z}\circ \tau)^n(\zeta)
\;=\;\sum_{n=1}^\infty \,
b^{(k)}_n (e^{-i\alpha}\widehat{Z}\circ \tau)^n(e^{i\alpha}\zeta)\\
&=\;\sum_{n=1}^\infty \,
a^{(k)}_n (\widehat{Z}\circ \tau)^n(e^{i\alpha}\zeta)
\;=\;\sum_{n=1}^\infty \,
a^{(k)}_n (\widehat{Z}\circ L)^n(\zeta)\,.
\end{split}
\end{equation*}
It follows from \eqref{eq3.6} that
\begin{equation*}
(q_k\circ L)'(0) = a^{(k)}_1 (\widehat{Z}\circ L)'(0)
= a^{(k)}_1 e^{i\alpha}\,,
\end{equation*}
so that $(\widehat{Z}\circ L)'(0) = e^{i\alpha}$. Then
Schwarz's lemma shows that $\widehat{Z}\circ L(\zeta)
= e^{i\alpha}\zeta$, which is the desired equation
\eqref{eq3.4}.
\end{proof}

\medskip
We notice that the unimodular constant $e^{i\alpha}$
above does not depend on $q_k$, because of our
assumption $(q_k\circ \tau)'(0) = a^{(k)}_1> 0$ in
Lemma \ref{lem2.1}.

\medskip
\begin{lem}
\label{lem3.4}
Let $L$ and $\{L^{(m)}\}$ be as above. Then for all $f$
in $B(\Delta)$,
\begin{equation}
\label{eq3.7}
\lim_{m\to \infty}\,f\circ L^{(m)}(\zeta) = f\circ L(\zeta)
\end{equation}
uniformly on compact subsets of $\Delta$.
\end{lem}

\begin{proof}
\; Let $Z$ be the inner function in Lemma \ref{lem3.3}, and
let $Z_0 = e^{-i\alpha} Z$ for the unimodular constant
$e^{i\alpha}$ in \eqref{eq3.4}. Then $\widehat{Z}_0\circ
L(\zeta) = \zeta$.  Let $\{q_k\}$ be the sequence in
Lemma \ref{lem2.1}, and put $a_n^{(k)} = \langle q_k,
Z^{\,n}_0\rangle$ in $L^2(\mu)$. Then we obtain
\begin{equation*}
q_k\circ L(\zeta)
= a_1^{(k)}\zeta + a_2^{(k)}\zeta^2 + \cdots \,,
\qquad \zeta \in \Delta\,,
\end{equation*}
with $a_1^{(k)} = (q_k\circ L)'(0)= |a_1^{(k)}|\,
e^{i\alpha}$. Notice that $a_1^{(k)}$ is different
from the one in the proof of Lemma \ref{lem3.3}.
Moreover, since
\begin{equation*}
\Vert q_k - Z \Vert_{L^2(\mu)} \rightarrow 0
\qquad \text{and} \qquad
1 \;\geq \;\Vert e^{-i\alpha} q_k \Vert_{L^2(\mu)}^2
\;\geq \; |a_1^{(k)}|^2 +
\sum_{n=2}^\infty\, |a_n^{(k)}|^2 \,,
\end{equation*}
we see that $a_1^{(k)} \to e^{i\alpha}, \,q_k\circ
L(\zeta) \to \widehat{Z}_0\circ L(\zeta) = \zeta$ and
$\sum_{n=2}^\infty\, |a_n^{(k)}|^2 \,\to 0$. Since
$a_1^{(k)} = (q_k\circ L)'(0)$ , \eqref{eq3.6} implies
also that $(q_k\circ L^{(m)})'(0) \to e^{i\alpha}$,
whenever $ k,\, m \to \infty$ with $ k \leq m$.
On the other hand, since $\{L^{(m)}\}$ is a sequence
in the compact space $\mathfrak{M}(H^\infty)^\Delta$,
$\{L^{(m)}\}$ has adherent points. To obtain
\eqref{eq3.7}, it suffices to show that any adherent
point must equal $L$. Let $(L^{(m_\gamma)})$ be an
arbitrary subnet of $\{L^{(m)}\}$ converging to $L_1$
in $\mathfrak{M}(H^\infty)^\Delta$.  Precisely for
all $f$ in $B(\Delta)$,
\begin{equation*}
\lim_\gamma \, f\circ
L^{(m_\gamma)}(\zeta) \; =\; f\circ L_1(\zeta)
\end{equation*}
uniformly on compact subsets of $\Delta$. It follows from
\eqref{eq3.5} that
\begin{equation*}
q_k\circ L_1(\zeta) = \lim_\gamma\, q_k\circ
L^{(m_\gamma )}(\zeta) = q_k\circ L(\zeta).
\end{equation*}
This shows that $|(q_k\circ L_1)'(0)| = |a_1^{(k)}|
\to 1$ as $k\to \infty$. So we see that $L_1$ is an
analytic map from $\Delta$ to an analytic disc
$P(\phi_1)$ with $L_1(0) = \phi_1$.
Let $Z_1, \widehat{Z}_1$ and $\tau_1 = \widehat{Z}_1^{-1}$
be as above by Wermer's embedding theorem on $P(\phi_1)$.
Denote by $\mu_1$ the representing measure for $\phi_1$,
and put $b_n^{(k)}=\langle q_k, Z^{\,n}_1\rangle$ in
$L^2(\mu_1)$. Then $\{b_n^{(k)}\}$ is a bounded sequence
in $\ell^2$ and $q_k$ is represented as
\begin{equation*}
q_k(\xi)
= b_1^{(k)}\widehat{Z}_1(\xi) + b_2^{(k)}
\widehat{Z}_1(\xi)^2 + \cdots\,,
\qquad \xi \in P(\phi_1).
\end{equation*}
Observe that
\begin{equation*}
a_1^{(k)} = (q_k\circ L)'(0)\;=\; (q_k\circ L_1)'(0)
\;=\;b_1^{(k)}(Z_1\circ L_1)'(0).
\end{equation*}
Since $|a_1^{(k)}| \to 1$ and $|(Z_1\circ L_1)'(0)| \leq 1$,
we see that $|(Z_1\circ L_1)'(0)| =1$. So multiplying $Z_1$
by unimodular constant, we may assume that $Arg(a_1^{(k)}) 
= Arg(b_1^{(k)})$. Since $q_k\circ L(\zeta) = q_k\circ L_1(\zeta)$
by \eqref{eq3.5}, we have
\begin{equation*}
\sum_{n=1}^{\infty}\, a_n^{(k)} \zeta^n
\;=\; q_k\circ L(\zeta)
\;=\; q_k\circ L_1(\zeta)
\;=\;\sum_{n=1}^{\infty}\, b_n^{(k)} \zeta^n,
\qquad \zeta \in \Delta\,,
\end{equation*}
which implies that $a_n^{(k)} = b_n^{(k)}$ for all
$n = 1, 2, \cdots$. Especially, $a_1^{(k)} = b_1^{(k)}
\to e^{i\alpha}$ as $k \to\infty$.
Let us show that $L_1\,=\,L$ in $\mathfrak{M}
(H^\infty)^\Delta$. Indeed, let $f$ be a function
in $B(\Delta)$, and let $a_n= \langle f, Z^n_0\rangle$
in $L^2(\mu)$. Then $f$ is represented as \eqref{eq2.1}.
Since $||q_k||_\infty \leq 1$ and $\Vert e^{-i\alpha}
q_k - Z_0 \Vert_{L^2(\mu)} \to 0$, we see easily that
\begin{equation*}
f(\xi)
\;=\;\sum_{n=0}^{\infty}\, a_n \widehat{Z}_0(\xi)^n
\;=\;\lim_{k\to\infty} \sum_{n=0}^{\infty}\, a_n \,
(e^{-i\alpha} q_k)^n(\xi), \qquad \xi \in P(\phi).
\end{equation*}
Since $\{a_n\}$ is a sequence in $\ell^2$, the function
$g\,=\,\sum_{n=0}^{\infty}\, a_n\,Z_1^n$ lies in
$H^2(\mu_1)$. the closure of $B(X)$ in $L^2(\mu_1)$.
Since $b_1^{(k)} \to e^{-i\alpha}$, we see $\Vert
e^{-i\alpha}q_k - Z_1\Vert_{L^2(\mu_1)}\to 0$.
Although $g$ may not equal $f$ in $L^2(\mu_1)$,
we obtain
\begin{equation*}
f(\xi)
\;=\;g(\xi)
\;=\;\sum_{n=0}^{\infty}\, a_n \widehat{Z}_1(\xi)^n
\;=\;\lim_{k\to\infty} \;\sum_{n=0}^{\infty}\, a_n \,
(e^{-i\alpha} q_k)^n(\xi), \qquad \xi \in P(\phi_1).
\end{equation*}
This shows that
\begin{equation*}
f\circ L_1(\zeta)
\;=\;\sum_{n=0}^{\infty}\, a_n \zeta^n
\;=\; f\circ L(\zeta), \qquad \zeta \in \Delta,
\end{equation*}
for all $f$ in $B(\Delta)$,  Since any adherent point
$L_1$ of $\{L^{(m)}\}$ always equals $L$
in $\mathfrak{M}(H^\infty)^\Delta$, we obtain
the equality \eqref{eq3.7}, as desired.
\end{proof}

\medskip
Let us make some comments on Lemma \ref{lem3.4}. Since
$\{w_i^m\}$ may be replaced by any of its subsequence,
so there are a lot of adherent points being considered
as $L^{(m)}$. However, since the sequence
$\{L^{(m)}\}$ converges to $L$, it follows from
\cite[165p, Theorem]{H1} that all but a finite number
of the parts $P(\phi^{(m)}) = L^{(m)}(\Delta)$ lie in
the same fiber $\mathfrak{M}_1$.

\medskip
\begin{lem}
\label{lem3.5}
Let $\{w_i^m\}$ be as above. Then, for each $m$, there
is a subsequence of $\{c_i^m\}$ of $\{w_i^m\}$ satisfying
that the sequence,
\begin{equation}
\label{eq3.8}
\{c_1, c_2, \cdots \} \;=\; \bigcup_{m=1}^{\infty}\{c_i^m\}\,,
\end{equation}
is an interpolating sequence in $\Delta$.
\end{lem}

\medskip
By the conformal invariant character of interpolating sequences, we discuss the proof of Lemma \ref{lem3.5}
in the upper half plane $\mathbf{H}$. Let
\begin{equation}
\label{eq3.9}
z(w)\;= \;i\,\frac{1-w}{1+w}\,.
\end{equation}
Then $z=z(w)$ is a linear fractional transformation
of $\overline{\Delta} = \{\vert w\vert\leq 1\}$ onto
$\overline{\mathbf{H}} = \{Im \, w \geq 0\}$
with $z(1)= 0$. We need the following characterization
of interpolating sequences (\cite[Chapter VII,
Theorem 1.1]{Ga}): A sequence $\{z_j\}$ in $\mathbf{H}$
is an interpolating sequence if and only if the points
$z_j=x_j+iy_j$ is separated, that is,
\begin{equation*}
\rho(z_j,z_k)\; =\; \left\vert \frac{z_j-z_k}{z_j-\overline{z}_k}\right\vert
\;\geq\; a > 0 \,,\qquad k\neq j\,,
\end{equation*}
and there is a constant $A$ such that
for every square
$Q= \{x_0 \leq x \leq x_0+\ell(Q),\; 0<y \leq\ell(Q)\,\}$,
\begin{equation*}
\sum_{z_j \in Q}\, y_j \;\leq A \ell(Q)\,,
\end{equation*}
that is, $\sum y_j\delta_{z_j}$ is a Carleson measure
on $\mathbf{H}$.

\bigskip
\noindent
{\itshape Proof of Lemma 3.5.\/} For the above sequence
$\{w_i^m\}$, we put $z_i^m\,=\,z(w_i^m)$ by \eqref{eq3.9}.
Then for each $m$, $\{z_i^m\}$ is a thin interpolating
sequence in $\mathbf{H}$ such that $z_i^m=x_i^m+iy_i^m \rightarrow 0$, with $y_i^m > 0$. We now choose an
interpolating subsequence $\{z_i\}$ of
$\cup_{m=1}^\infty \{z_i^m\}$. Let $z_1 = z_1^1 = x_1^1+i y_1^1$
and $z_2 = z_{i_2}^1$ with $y_{i_2}^1 \leq \frac{1}{2}y_1^1$.
We then put $z_3 = z_{i_3}^2$ such that $y_{i_3}^2 \leq \frac{1}{2}y_{i_2}^1$ in $\{z^2_j\}$, and $z_4 = z_{i_4}^1$ such that $y_{i_4}^1 \leq \frac{1}{2}y_{i_3}^2$ in $\{z^1_j\}$. The later
elements $z_k=x_{i_k}^m+iy_{i_k}^m$ are defined recursively,
according to the following array
\begin{equation*}
\begin{matrix}
\{c^1_i\}: &z_1, &z_2, &z_4, &z_7, &z_{11},&\cdots&  & &\in \{z^1_i\}\\
\{c^2_i\}: & &z_3, &z_5, &z_8, &z_{12},&\cdots& &  &\in \{z^2_i\}\\
\{c^3_i\}: & & &z_6, &z_9, &z_{13},&\cdots& &  &\in \{z^3_i\}\\
\{c^4_i\}: & & & &z_{10}, &z_{14},&\cdots& &  &\in \{z^4_i\}\\
\{c^5_i\}: & & & & &z_{15} &\cdots & &  & \in \{z^5_i\}\\
 &\hdotsfor{8}
\end{matrix}
\end{equation*}
which have the following properties:

\begin{enumerate}
\renewcommand{\labelenumi}{(\alph{enumi})}

\medskip
\item Each $\{c^m_i\}$ is the thin interpolating sequence in
  $\Delta$ obtained by $c^m_i = w(z_k)$ with m-th row, where
  $z = w(z)$ is the inverse transformation of \eqref{eq3.9}.

\medskip
\item
  The sequence of the first elements of rows is
  the diagonal of the array, and if $z_k$ is such one
  of the $(m-1)$-th row, then $m$-th row begins $z_{k+m}$.

\medskip
\item
  For each $m$, $\{c^m_i\}$ is an infinite subsequence
  of $\{w_i^m\}$, so the role of $\{w_i^m\}$ in Lemma \ref
  {lem3.4} may be replaced by $\{c^m_i\}$.
\end{enumerate}
Let us show the sequence $\{c_i\}$ by \eqref{eq3.8} is an interpolation sequence. Recall that the $\{z_j=x_j+iy_j\}$
is separated if and only if $\vert z_j- z_k\vert\,\geq \,
b\, y_j$ for $k\neq j$, with $b>0$. Since $y_k < \frac{1}{2}y_j$ if $k>j$, $\{z_j\}$ satisfies this inequality with $b=\frac{1}{2}$. On the other hand, let
$Q= \{x_0 \leq x \leq x_0+\ell(Q),\; 0<y \leq\ell(Q)\,\}$.
For $\{z_j=x_j+iy_j\}$, the sequence $\{y_j\}$ is a decreasing
sequence with $\frac{1}{2}y_j > y_{j+1}$. Let $y_k$ be the
largest number with $y_j \leq \ell(Q)$. Then we easily
observe that
\begin{equation*}
\sum_{z_j \in Q}\, y_j
\;\leq \; \sum_{j=0}^\infty\,\left(\frac{1}{2}\right)^j y_k
\;\leq \; 2\,\ell(Q)\,.
\end{equation*}
Thus  $\{z_j\}$ is an interpolation sequence in
$\mathbf{H}$, so is $\{c_j\}$ in $\Delta$.
\hfill $\square$

\medskip
We are now ready for the proof of Theorem.

\medskip
\noindent
{\itshape Proof of Theorem.\/}
Lemmas \ref{lem3.3} and \ref{lem3.4} shows that (a) implies
(b). Since $\{w_i^m\}$ in Lemma \ref{lem3.4} may be replaced
by its subsequence, it follows from Lemma \ref{lem3.5} that
(b) implies (d). We next show that (d) implies (a). Let
$(c_\alpha)$ be a subnet of $\{c_i\}$ such that
$\lim_\alpha\,c_\alpha = \phi$ in $\mathfrak{M}(H^\infty)$.
Since $\{c_i\}$ is an interpolating sequence, the Hoffman
map $L = \lim_\alpha L_{c_\alpha}$ exists automatically
in $\mathfrak{M}(H^\infty)^\Delta$. Let $B$ be the
interpolating Blaschke product with zeros $\{c_i\}$. Then
we have
\begin{equation*}
\vert (B\circ L)'(0)\vert \;=\;
\lim_\alpha\, \vert (B\circ L_{c_\alpha}\,)'(0)\vert
\;=\; \lim_\alpha \,(1-|c_\alpha|^2)\,|B'(c_\alpha)|
\;\geq\; \delta \,>\,0.
\end{equation*}
So $B\circ L$ is a nonconstant analytic function on $\Delta$
with $B\circ L(0)=B(\phi)$. Since $P(\phi)$ contains
$L(\Delta)$ by \cite[Chapter X, Lemma 1.1]{Ga}, $P(\phi)$
is an analytic disc, thus (a) holds. On the other hand, it
follows from Lemma \ref{lem2.3} that (a) implies (c).
Conversely, suppose that (c) holds. Observe that the
measure
\begin{equation*}
d\widetilde{\mu}
\;=\; \frac{1-|\zeta|^2}{|F-\zeta|^2}\, d\mu
\end{equation*}
is mutually absolutely continuous with respect to $\mu$.
We show that $\widetilde{\mu}$ is a representing measure
for $B(\Delta)$. Since $F$ is an inner function, we see
that
\begin{equation*}
\frac{1-|\zeta|^2}{|F -\zeta|^2}
\;=\;\frac{1-|\zeta|^2}{|1-\overline{\zeta}\, F|^2}
\;=\; \sum_{n=1}^\infty \,\overline{\zeta}^n \,F^n
+ \sum_{n=0}^\infty \,\zeta^n \,\overline{F}^n\,.
\end{equation*}
For each $f$ in $H^1(\mu)$, we put $a_n(f) =
\int_X \,f \,\overline{F}^n\,d\mu$. We then fixe a
point $\zeta$ in $\Delta$ and define
\begin{equation}
\label{eq3.10}
\Phi_{\zeta}(f)
\;=\;\int_X f\, d\widetilde{\mu}
\;=\;\sum_{n=0}^\infty a_n(f)\,\zeta^n \,, \qquad
f \in H^1(\mu),
\end{equation}
which is a linear mapping. Since $H^2(\mu)=\mathcal{H} \oplus\mathcal{S}$, where $\mathcal{S}$ is orthogonal
to the system $\{F^{n}; \;n=0, \pm 1, \pm 2, \cdots\}$
in $L^2(\mu)$. This shows that $\Phi_{\zeta}(h)
\equiv 0$ on $\Delta$, for all $h$ in $\mathcal{S}$.
On the other hand, for
each $f$ in $H^2(\mu)$, there is a sequence $\{f_k\}$
in $B(\Delta)$ such that $||f_k - f||_{L^2(\mu)} \to 0$.
Since $\mathcal{S}$ is an invariant subspace, we observe
that $\Phi_{\zeta}(f_k h) \equiv 0$ on $\Delta$. It
follows from the Schwarz inequality that $a_n(f_k h) \to
a_n(fh)$, so $a_n(fh) = 0$ for all $k$. Then
\eqref{eq3.10} shows that $\Phi_{\zeta}(f h) \equiv 0$
on $\Delta$. Each $f$ in $H^2(\mu)$ has the form $f=g+h$
with $g$ in $\mathcal{H}$ and $h$ in $\mathcal{S}$. Since
$\zeta \to \Phi_{\zeta}(g) = \sum_{n=0}^\infty a_n(g)\,
\zeta^n$ is analytic on $\Delta$ and $g =
\sum_{n=0}^\infty a_n(g) F^n$, we see easily that
$\Phi_{\zeta}(g^2) = \Phi_{\zeta}(g)^2$. Then we have
\begin{equation*}
\Phi_{\zeta}(f^2)\;=\;\Phi_{\zeta}((g+h)^2)
\;=\;\Phi_{\zeta}(g^2)\;=\;\Phi_{\zeta}(g)^2
\;=\;\Phi_{\zeta}(f)^2,
\end{equation*}
from which we see that the probability measure
$\widetilde{\mu}$ is multiplicative on $H^2(\mu)$. Since
$\widetilde{\mu}$ is a representing measure for $B(\Delta)$
being different from $\mu$, $P(\phi)$ must be an analytic disc, so (a) holds. Thus the proof is complete.
\hfill $\square$

\bigskip
\section{Remarks}
\label{S2}
(a) Some properties on Gleason parts hold in more general
settings. Let $A$ be a uniform algebra on compact Hausdorff
space $Y$, and let $\mathfrak{M}(A)$ be the maximal ideal
space of $A$. Suppose that $\phi$ in $\mathfrak{M}(A)$ has
a unique representing measure $\mu$ on $Y$. If the Gleason
part $P(\phi)$ of $\phi$ is not a single point, then
Wermer's embedding theorem assures the existence of an
inner function $Z$ satisfying all of its properties. With 
this $Z$, the results in Section 2 hold as well. Moreover, 
by the same way as in the proof of Theorem, we obtain the 
equivalence of (a) and (c) of Theorem in these settings. It 
should be noted that if a Dirichlet algebra $A$ is contained 
in $C(Y)\cap \mathcal{H}(Z)$ on any analytic disc, then
$\mathcal{S}(Z) =\{0\}$. However, we do not know whether
this condition characterises Dirichlet algebras within 
logmodular algebras.

\medskip
(b) It would be useful to regard the inner function $Z$ as
a transformation of $X$ to $\mathbf{T}$. Let $\mathcal{C}$
be the class of all $Z^{-1}(E)$ for all Borel sets $E$ in
$\mathbf{T}$. Then $\mathcal{C}$ is a sub-$\sigma$-algebra
of all Borel sets $\mathcal{B}$ in $X$. For an $f$ in
$L^1(\mu) = L^1(X,\mathcal{B},\mu)$, $E(f|\mathcal{C})$
denotes the conditional expectation of $f$ with respect
to $\mathcal{C}$ (see \cite[Chapter 1, 1.4 D]{P}). On
$L^2(Y,\mathcal{B},\mu)$, $E(\cdot\,|\mathcal{C})$ is the
orthogonal projection $Pr_{\mathcal{C}}$ to the subspace
$L^2(Y,\mathcal{C},\mu)$. We notice that
$Pr_{\mathcal{C}}(H^2(\mu)) = \mathcal{H}(Z)$ and that
$E(\cdot\,|\mathcal{C})$ maps $L^\infty(Y,\mathcal{B},
\mu)$ to $L^\infty(Y,\mathcal{C},\mu)$. Recall that each
$f$ in $B(\Delta)$ has the form $f=g+h$ with $g$ in
$\mathcal{H}$ and $h$ in $\mathcal{S}$. Since $B(\Delta)$
is contained in $H^\infty(\mu)$, we observe that both $g$
and $h$ belong to $H^\infty(\mu)$. This enables us to
omit the approximation process from the proof of
Theorem.

\end{document}